\documentclass{amsart}
\usepackage{amssymb,amstext,amsmath,amscd,amsthm,amsfonts,enumerate,graphicx,latexsym,stmaryrd,multicol}

\usepackage[usenames]{color}
\usepackage[all]{xy}

\usepackage[top=34truemm,bottom=15truemm,left=25truemm,right=25truemm]{geometry}
\newtheorem{thm}{Theorem}[section]
\newtheorem{lem}[thm]{Lemma}
\newtheorem{prop}[thm]{Proposition}
\newtheorem{cor}[thm]{Corollary}
\theoremstyle{definition}
\newtheorem{dfn}[thm]{Definition}

\newtheorem{rmk}[thm]{Remark}
\theoremstyle{remark}

\newtheorem*{ac}{Acknowlegments}

\numberwithin{equation}{thm}
\def\cm{\mathsf{CM}}
\def\Cok{\operatorname{Coker}}

\def\D{\mathrm{D}}

\def\Ext{\operatorname{Ext}}

\def\fl{\mathsf{FL}}
\def\G{\mathcal{G}}

\def\ge{\geqslant}
\def\gp{\mathsf{GP}}
\def\grade{\operatorname{grade}}
\def\grd{\mathsf{Grd}}
\def\H{\operatorname{H}}

\def\Hom{\operatorname{Hom}}

\def\image{\operatorname{Im}}

\def\le{\leqslant}

\def\m{\mathfrak{m}}
\def\mod{\operatorname{mod}}

\def\Perf{\mathsf{Perf}}
\def\proj{\operatorname{\mathsf{proj}}}

\def\Ref{\mathsf{Ref}}

\def\sph{\mathsf{Sph}}

\def\syz{\Omega}

\def\tf{\mathsf{TF}}

\def\tr{\operatorname{Tr}}

\def\X{\mathcal{X}}

\def\Y{\mathcal{Y}}
\def\ZZ{\mathbb{Z}}
\begin{document}
\allowdisplaybreaks
\title[Stable categories of spherical modules and torsionfree modules]{Stable categories of spherical modules and torsionfree modules}
\author{Yuya Otake}
\address{Graduate School of Mathematics, Nagoya University, Furocho, Chikusaku, Nagoya 464-8602, Japan}
\email{m21012v@math.nagoya-u.ac.jp}

\begin{abstract}
Auslander and Bridger introduced the notions of $n$-spherical modules and $n$-torsionfree modules.
In this paper, we construct an equivalence between the stable category of $n$-spherical modules and the category of modules of grade at least $n$, and provide its Gorenstein analogue.
As an application, we prove that if $R$ is a Gorenstein local ring of Krull dimension $d>0$, then there exists a stable equivalence between the category of $(d-1)$-torsionfree $R$-modules and the category of $d$-spherical modules relative to the local cohomology functor. 
\end{abstract}
\maketitle
%\tableofcontents
%%%%%%%%%%%%%%%%%%%%%%%%%%%%%%%%%%%%%%%%%%%%%%%%%%%%%%%%%%%%
\section{Introduction}
Throughout this paper, let $R$ be a two-sided noetherian ring.
All subcategories are assumed to be strictly full.
Denote by $\mod R$ the category of finitely generated (right) $R$-modules.

Auslander and Bridger \cite{AB} introduced the notion of $n$-spherical modules for each positive integer $n$: a finitely generated  $R$-module $M$ is called {\em $n$-spherical} if $\Ext^i_R(M,R)=0$ for all $1\le i\le n-1$ and $M$ has projective dimension at most $n$.
Note that when this is the case, $\Ext^i_R(M,R)=0$ for all $i\ne0, n$.
Auslander and Bridger found various important properties related to $n$-spherical modules.
For example, the spherical filtration theorem they proved asserts that any finitely generated module $M$ satisfying a certain grade condition has a filtration
$$M_m\subset M_{m-1}\subset\cdots\subset M_1\subset M_0=M\oplus P,$$
where $P$ is projective, such that $M_{j-1}/M_j$ is $j$-spherical for all $1\le j\le m$.
Recently, Huang \cite{H} proved the dual version of the spherical filtration theorem.

In this paper, we study the stable category of $n$-spherical modules.
Moreover, we introduce the notion of $n$-G-spherical modules by replacing projective dimension in the definition of $n$-spherical modules with Gorenstein dimension, and give similar results for the stable category of $n$-G-spherical modules.
These are related to the category of modules with high grade and the category of totally reflexive modules.
To be precise, the following theorem holds.

\begin{thm}\label{main1}
Let $n$ be a positive integer.
Consider the following subcategories of $\mod R$.
\begin{align*}
\sph_n(R)&=\{M\in \mod R\mid\text{$M$ is $n$-spherical }\},\\
\grd_n(R)&=\{M\in \mod R\mid \grade_R M\ge n\},\\
\sph_n^{\sf G}(R)&=\{M\in \mod R\mid\text{$M$ is $n$-G-spherical }\},\\
\Ref^{\sf T}(R)&=\{M\in \mod R\mid\text{$M$ is totally reflexive }\}.
\end{align*}
One then has the equivalences
$$
\xymatrix{
\underline{\sph_n(R)}\ar@<.5mm>[rr]^-{\Ext^n(-,R)}&&\grd_n(R^{\rm op})\ar@<.5mm>[ll]^-{\tr\syz^{n-1}},\\
\underline{\sph_n^{\sf G}(R)}\ar@<.5mm>[rr]^-{\tr\syz^{n-1}}&&\underline{\grd_n(R^{\rm op})\ast\Ref^{\sf T}(R^{\rm op})}.\ar@<.5mm>[ll]^-{\tr\syz^{n-1}}
}
$$
\end{thm}

Here, $\syz(-)$ and $\tr(-)$ respectively stand for the syzygy and the (Auslander) transpose, while the stable category of a subcategory $\X$ of $\mod R$ is denoted by $\underline\X$.
For two subcategories $\X$ and $\Y$ of $\mod R$, we denote by $\X\ast\Y$ the subcategory of $\mod R$ consisting of modules $M$ such that there is an exact sequence $0\to L\to M\to N\to0$ with $L\in\X$ and $N\in\Y$.

The notion of $n$-torsionfree modules was also introduced by Auslander and Bridger \cite{AB}, and played a central role in the stable module theory they developed.
The structure of $n$-torsionfree modules has been well-studied; see \cite{AB, DT, EG, MTT}.
In the following, applying Theorem \ref{main1} to the case where $(R,\m)$ is a commutative local ring, we describe the structure of $n$-torsionfree modules.
Denote by $\fl(R)$ the subcategory of $\mod R$ consisting of modules of finite length, and by $\H^i_{\m}(-)$ the $i$-th local cohomology functor with respect to $\m$.
Also, we say that a finitely generated $R$-module $M$ is {\em $n$-H-spherical} if $\H^i_{\m}(M)=0$ for all $i\ne0, n$.

\begin{cor}\label{main2}
Suppose that $R$ is commutative, local and with Krull dimension $d>0$.
Let $\m$ be the maximal ideal of $R$.
For a nonnegative integer $n$, consider the following subcategories of $\mod R$.
\begin{align*}
\tf_n(R)&=\{M\in \mod R\mid\text{$M$ is $n$-torsionfree }\},\\
\sph_n^{\sf H}(R)&=\{M\in \mod R\mid\text{$M$ is $n$-H-spherical }\}.
\end{align*}
\begin{enumerate}[\rm(1)]
    \item
    If $R$ is regular, then one has the equivalence
    $$
    \xymatrix@R-1pc@C-1pc{
    \underline{\tf_{d-1}(R)}\ar@<.5mm>[rrrrr]^-{\Ext^d(\tr(-),R)}&&&&&
    \fl(R).\ar@<.5mm>[lllll]^-{\syz^{d-1}}
    }
    $$
    \item
    If $R$ is Gorenstein, then one has the equivalence
    $$
    \xymatrix@R-1pc@C-1pc{
    \underline{\tf_{d-1}(R)}\ar@<.5mm>[rrrrr]^-{\tr\syz^{d-1}\tr}&&&&&
    \underline{\sph_{d}^{\sf H}(R)}.\ar@<.5mm>[lllll]^-{\syz^{d-1}}
    }
    $$
\end{enumerate}
\end{cor}

It may be well-known to experts that when $R$ is a three dimensional regular local ring, there is an equivalence
$$
\xymatrix@R-1pc@C-1pc{
\underline{\Ref(R)}\ar@<.5mm>[rrrr]^{\Ext^1((-)^\ast,R)}&&&&\fl(R)\ar@<.5mm>[llll]^{\syz^2},
}
$$
where $\Ref(R)$ stands for the subcategory of $\mod R$ consisting of reflexive $R$-modules, while $(-)^\ast$ denotes the $R$-dual.
The above corollary gives a higher dimensional version of this result.

%%%%%%%%%%%%%%%%%%%%%%%%%%%%%%%%%%%%%%%%%%%%%%%%%%%%%%%%%%%%%%%%%

\section{Our results and proofs}
In this section we give several definitions, and state and prove our results.
We also give proofs of Theorem \ref{main1} and Corollary \ref{main2}, which are displayed in the previous section.
We begin with recalling fundamental notions. 
\begin{dfn}\label{basicdfn}
\begin{enumerate}[\rm(1)]
    \item
    Let $\X$ be a subcategory of $\mod R$.
    The {\em stable category} $\underline{\X}$ of $\X$ is defines as follows: The objects of $\underline{\X}$ are the same as those of $\X$.
    The morphism set of objects $X, Y$ of $\underline{\X}$ is the quotient of the additive group $\Hom_R(X,Y)$ by the subgroup consisting of $R$-homomorphisms factoring through some finitely generated projective $R$-modules.
    Note that $\underline{\X}$ is none other than the subcategory of $\underline{\mod R}$ consisting of objects $M$ such that $M\in\X$, that is,
    $$
    \underline{\X}=\{ M\in\underline{\mod R} \mid M\in\X\}.
    $$
    \item
    Let $M$ be a finitely generated $R$-module and $P_1\xrightarrow{d_1}P_0\xrightarrow{d_0}M\to0$ a finite projective presentation of $M$.
    The {\em syzygy} $\syz M$ of $M$ is defined as $\image d_1$.
    Note that $\syz M$ is uniquely determined by $M$ up to projective summands.
    Taking the syzygy induces an additive functor $\syz:\underline{\mod R}\to\underline{\mod R}$.
    Inductively, we define $\syz^n=\syz\circ\syz^{n-1}$ for an integer $n>0$.
    The {\em (Auslander) transpose} $\tr M$ of $M$ is defined as $\Cok d_1^\ast$.
    Note that $\tr M$ is uniquely determined by $M$ up to projective summands.
    Taking the transpose induces an additive functor $\tr:\underline{\mod R}\to\underline{\mod R^{\mathrm{op}}}$.
    \item
    Let $m,n\in\ZZ_{\ge0}\cup\{\infty\}$.
    We denote by $\G_{m,n}(R)$, or simply $\G_{m,n}$, the subcategory of $\mod R$ consisting of $R$-modules $M$ such that $\Ext^i_R(M,R)=0$ for all $1\le i\le m$ and $\Ext^j_{R^{\mathrm{op}}}(\tr M,R)=0$ for all $1\le j\le n$.
    We denote by $\proj(R)$ (resp. $\gp(R)$) the subcategory of $\mod R$ consisting of finitely generated projective (resp. Gorenstein projective) $R$-modules.
    Note that $\gp(R)=\G_{\infty, \infty}$.
    A finitely generated $R$-module $M$ is called {\em $n$-torsionfree} if $M$ belongs to $\G_{0,n}$.
    We denote by $\tf_n(R)$ the subcategory of $\mod R$ consisting of $n$-torsionfree modules, that is, we set $\tf_n(R)=\G_{0,n}$.
    \item
    The {\em projective dimension} (resp. {\em Gorenstein dimension}) of a finitely generated $R$-module $M$ is defined to be the infimum of integers $n$ such that there exists an exact sequence
    $$
    0\to X_n\to X_{n-1}\to\cdots\to X_1\to X_0\to M\to0
    $$
    of finitely generated $R$-modules with $X_i$ projective (resp. Gorenstein projective).
    \item
    The {\em grade} of a finitely generated $R$-module $M$ is defined to be the infimum of integers $i$ such that $\Ext_R^i(M,R)=0$, and denoted by $\grade_R M$.
    For an integer $n$, we denote by $\grd_n(R)$ the subcategory of $\mod R$ consisting of $R$-modules $M$ satisfying that $\grade_R M\ge n$.
    \item
    Let $M, N$ be finitely generated $R$-modules.
    We say that $M$ and $N$ are {\em stably isomorphic} if there are finitely generated projective modules $P, Q$ such that $M\oplus P\cong N\oplus Q$, and then write $M\approx N$.
    Note that $M\approx N$ if and only if $M$ and $N$ are isomorphic as objects of $\underline{\mod R}$.
    \item
    Let $\X$ be a subcategory of $\mod R$.
    We say that $\X$ is {\em closed under stable isomorphism} if for finitely generated modules $M, N$ with $M\in\X$ and $M\approx N$, it holds that $N\in\X$.
\end{enumerate}
\end{dfn}

We extend the definition of $n$-spherical modules due to Auslander and Bridger.
Namely, for any subcategory $\X$ of $\mod R$ closed under stable isomorphism, we introduce the concept of $n$-$\X$-spherical modules as follows.

\begin{dfn}\label{sph}
Let $\X$ be a subcategory of $\mod R$.
Suppose that $\X$ is closed under stable isomorphism.
Let $n\ge1$ be an integer and $M$ a finitely generated $R$-module.
We say that $M$ is {\em $n$-$\X$-spherical} if $\Ext^i_R(M,R)=0$ for all $1\le i\le n-1$ and $\syz^n M\in\X$.
We denote by $\sph_n^{\X}(R)$ the subcategory of $\mod R$ consisting of $n$-$\X$-spherical $R$-modules.
We call $n$-$\proj(R)$-spherical (resp. $n$-$\gp(R)$-spherical) simply {\em $n$-spherical} (resp. {\em $n$-G-spherical}).
We denote by $\sph_n(R)$ (resp. $\sph_n^{\sf G}(R)$) the subcategory of $\mod R$ consisting of $n$-spherical (resp. $n$-G-spherical) modules.
\end{dfn}

\begin{rmk}\label{sphrmk}
Let $n\ge1$ be an integer and $M$ a finitely generated $R$-module.
Then $M$ is $n$-spherical if and only if $\Ext^i_R(M,R)=0$ for all $1\le i\le n-1$ and $M$ has projective dimension at most $n$.
Hence our convention is consistent with the original definition by Auslander and Bridger.
Similarly, $M$ is $n$-G-spherical if and only if $\Ext^i_R(M,R)=0$ for all $1\le i\le n-1$ and $M$ has Gorenstein dimension at most $n$.
\end{rmk}

For a finitely generated $R$-module $M$, we denote by $\D(M)$ the image of the canonical map $\sigma_M :M\to M^{\ast\ast}$ given by $\sigma_M(m)(f)=f(m)$ for $m\in M$ and $f\in M^\ast$.
This correspondence induces an additive functor $\D:\mod R\to\mod R$.
Note that $\D(M)\approx \syz\tr\syz\tr M$; see \cite[Appendix]{AB}.
The following proposition is an essential part of the proof of the main theorem.

\begin{prop}\label{key}
Let $n\ge1$ be an integer and $M$ a finitely generated $R$-module.
Let $\X$ be a subcategory of $\mod R$.
Assume that $\X$ satisfies the following conditions.
\begin{enumerate}[\quad\rm(i)]
    \item
    The subcategory $\X$ is closed under stable isomorphism.
    \item
    One has $\D(\X)\subset\X$.
    \item
    One has $\X\subset\G_{n,0}$.
\end{enumerate}
Then the following are equivalent.
\begin{enumerate}[\rm(1)]
    \item
    There exists an exact sequence $0\to L\to M\to N\to0$ with $\grade_R L\ge n$ and $N\in\X$.
    \item
    The module $M$ belongs to $\G_{n-1,0}$ and the module $\D(M)$ belongs to $\X$.
\end{enumerate}
In other words, one has $\grd_n(R)\ast\X=\D^{-1}(\X)\cap\G_{n-1,0}$.
\end{prop}

\begin{proof}
Suppose that there is an exact sequence $0\to L\to M\to N\to0$ with $L\in\grd_n(R)$ and $N\in\X$.
Then $M$ is in $\G_{n-1,0}$ as $\grd_n(R)$ and $\X$ are contained in $\G_{n-1,0}$.
Since $L^\ast=0$, by \cite[Lemma 3.9]{AB}, there is an exact sequence $0\to A\to B\to C\to0$ such that $A\approx\tr N$, $B\approx\tr M$ and $C\approx\tr L$.
As $\syz\tr L$ is projective, we have $\syz\tr M\approx\syz\tr N$.
Hence $\D(M)\approx\D(N)$.
We get $\D(N)\in\X$ by the assumption (ii), and we obtain $\D(M)\in\X$ by the assumption (i).
The implication $(1)\Rightarrow(2)$ holds.

Conversely, assume that $M\in\G_{n-1,0}$ and $\D(M)\in\X$.
We consider the exact sequence $0\to\Ext^1(\tr M,R)\to M\xrightarrow{\widetilde{\sigma_M}}\D(M)\to0$; see \cite[Proposition 2.6]{AB}.
Set $L=\Ext^1(\tr M,R)$ and $N=\D(M)$.
As ${\sigma_M}^\ast$ is surjective, so is ${\widetilde{\sigma_M}}^\ast$.
We obtain a long exact sequence
$$
\xymatrix@R-1pc@C-1pc{
    &0\ar[r]&L^\ast\ar[r]&\Ext^1(N,R)\ar[rr]&&\Ext^1(M,R)\ar[rr]&&\Ext^1(L,R)\ar[rr]&&\cdots\\
}.
$$
Now $M$ belongs to $\G_{n-1,0}$, and $N$ belongs to $\G_{n,0}$ by the assumption (iii).
Thus $L$ has grade at least $n$, and the implication $(2)\Rightarrow(1)$ holds.
\end{proof}

Here are some comments about Proposition \ref{key}.

\begin{rmk}\label{dirsum}
Let $m\ge n$ be integers.
\begin{enumerate}[\rm(1)]
    \item
    The subcategory $\proj(R)$ satisfies the three assumptions (i), (ii) and (iii) of Proposition \ref{key}, and so do the subcategories $\gp(R)$ and $\G_{m,m+1}(R)$; see \cite[Proposition 1.1.1]{I}.
    \item
    Let $\X$ be a subcategory of $\mod R$ satisfying the three assumptions (i), (ii) and (iii) of Proposition \ref{key}.
    By Proposition \ref{key}, if $\X$ is closed under direct summands, then so is $\grd_n(R) \ast \X$.
    Hence $\grd_n(R) \ast\gp(R)$ and $\grd_n(R) \ast\G_{m,m+1}(R)$ are closed under direct summands.
\end{enumerate}
\end{rmk}

The following lemma connects Proposition \ref{key} to the notion of $n$-$\X$-spherical modules.

\begin{lem}\label{cateq}
Let $n\ge1$ be an integer.
\begin{enumerate}[\rm(1)]
    \item
    Let $\X$ be a subcategory of $\mod R$.
    Assume that $\X$ is closed under stable isomorphism.
    One then has the duality
    $$
    \xymatrix@R-1pc@C-1pc{
    \underline{\D^{-1}(\X)\cap\G_{n-1,0}(R)}\ar@<.5mm>[rrr]^-{\tr\syz^{n-1}}&&&
    \underline{\{ M'\in\G_{n-1,0}(R^{\mathrm{op}}) \mid \syz\tr\syz^n M'\in\X\}}\ar@<.5mm>[lll]^-{\tr\syz^{n-1}}
    }.
    $$
    \item
    Let $M'$ be a finitely generated $R^{\mathrm{op}}$-module and $m\ge0$ an integer.
    Let $\X(R)$ be any of the subcategories $\proj(R)$, $\gp(R)$ and $\G_{m,m+1}(R)$.
    Then $\syz\tr\syz^n M'$ belongs to $\X(R)$ if and only if $\syz^n M'$ belongs to $\X(R^{\mathrm{op}})$.
\end{enumerate}
\end{lem}

\begin{proof}
(1) By \cite[Proposition 1.1.1]{I}, the functor $\tr\syz^{n-1}:\underline{\mod R}\leftrightarrow\underline{\mod R^{\rm op}}$ gives a duality $\underline{\G_{n-1,0}(R)}\leftrightarrow\underline{\G_{n-1,0}(R^{\rm op})}$.
Hence it is enough to show that both of the above restricted correspondences are well-defined.
Let $M$ be a finitely generated $R$-module satisfying that $M\in\G_{n-1,0}(R)$ and $\D(M)\in\X$.
Set $M'=\tr\syz^{n-1}M$.
We have
$$
\syz\tr\syz^n M'\approx \syz\tr\syz\tr\tr\syz^{n-1}\tr\syz^{n-1}M\approx \syz\tr\syz\tr M\approx\D(M).
$$
Thus the functor from the left hand side is well-defined.
The converse is proved similarly.

(2) Since $\syz^n M'$ is in $\G_{0,1}(R^{\mathrm{op}})$, one has $\syz\tr\syz\tr\syz^n M'\approx\syz^n M'$.
Moreover, the inclusion $\syz\tr(\X(R))\subset\X(R^{\mathrm{op}})$ holds now.
Hence the assertion follows.
\end{proof}

The following theorem is the main result of this paper.

\begin{thm}\label{sphthm}
Let $n\ge1$ be an integer.
\begin{enumerate}[\rm(1)]
    \item
    One has the equivalence
    $$
    \xymatrix@R-1pc@C-1pc{
    \underline{\sph_n(R)}\ar@<.5mm>[rrrr]^-{\Ext^n(-,R)}&&&&
    \grd_n(R^{\rm op}).\ar@<.5mm>[llll]^-{\tr\syz^{n-1}}
    }
    $$
    \item
    Let $m\in\ZZ_{\ge n}\cup\{\infty\}$.
    One then has the equivalence
    $$
    \xymatrix@R-1pc@C-1pc{
    \underline{\sph_n^{\G_{m,m+1}}(R)}\ar@<.5mm>[rrrr]^-{\tr\syz^{n-1}}&&&&
    \underline{\grd_n(R^{\rm op})\ast\G_{m,m+1}(R^{\rm op})}.\ar@<.5mm>[llll]^-{\tr\syz^{n-1}}
    }
    $$
\end{enumerate}
\end{thm}

\begin{proof}
The assertion (2) follows from Proposition \ref{key}, Remark \ref{dirsum}(1) and Lemma \ref{cateq}.
It is easily seen that the category $\underline{\grd_n(R) \ast\proj(R)}$ is naturally equivalent to the category $\underline{\grd_n(R)}=\grd_n(R)$.
Moreover, since $n$-spherical modules have projective dimension at most $n$, taking $\tr\syz^{n-1}$ on $\sph_n(R)$ is the same as taking $\Ext^n(-, R)$.
Hence the assertion (1) is seen similarly.
\end{proof}

\begin{proof}[Proof of Theorem \ref{main1}]
The first assertion of Theorem \ref{main1} is the same as Theorem \ref{sphthm}(1).
The second assertion follows by letting $m=\infty$ in Theorem \ref{sphthm}(2).
\end{proof}

\begin{rmk}\label{perf}
In general, the grade of a non-zero finitely generated module is less than or equal to its projective dimension.
A finitely generated $R$-module $M$ is said to be {\em perfect} if $M=0$ or the grade of $M$ equals its projective dimension.
For a positive integer $n$, we denote by $\Perf_n(R)$ the subcategory of $\mod R$ consisting of perfect $R$-modules with grade $n$ or infinity.
Then the equality $\grd_n(R)\cap\sph_n(R)=\Perf_n(R)$ holds.
By restricting the correspondence of Theorem \ref{sphthm}(1), we obtain the duality $\Ext_R^n(-,R):\Perf_n(R)\xrightarrow{\sim}\Perf_n(R^{\mathrm{op}})$.
Theorem \ref{sphthm}(1) can be regarded as a generalization of this classical result.
\end{rmk}

In the rest of this paper, we assume that $R$ is commutative and local, and denote by $\m$ the maximal ideal of $R$.
Let $\Gamma_{\m}(-)$ be the $\m$-torsion functor; recall that $\Gamma_{\m}(M)=\{x\in M\mid {\m}^r x=0\text{ for some }r>0\}$ for an $R$-module $M$.
Let $\H_\m^i(-)$ be the $i$-th local cohomology functor, that is, the $i$-th right derived functor of $\Gamma_\m(-)$. 

A finitely generated $R$-module $M$ is said to be {\em maximal Cohen--Macaulay} if the depth of $M$ is greater than or equal to the (Krull) dimension of the ring $R$; see \cite[Chapter 2]{BH} for details.
We denote by $\cm(R)$ the subcategory of $\mod R$ consisting of maximal Cohen--Macaulay $R$-modules.
By Grothendieck's vanishing theorem \cite[Theorem 3.5.7]{BH}, a finitely generated $R$-module $M$ is maximal Cohen--Macaulay if and only if $\H^i_{\m}(M)=0$ for all $i<d$, where $d$ is the dimension of $R$.
The following proposition provides various equivalent conditions for a finitely generated module $M$ to be spherical relative to the local cohomology functor $\H_{\m}$.
Here, we say that a finitely generated $R$-module $M$ is {\em $n$-H-spherical} if $\H^i_{\m}(M)=0$ for all $i\ne0, n$. 

\begin{prop}\label{cm}
Suppose that $R$ is Cohen--Macaulay and with dimension $d>0$. 
Let $M$ be a finitely generated $R$-module.
Consider the following four conditions.
\begin{enumerate}[\quad\rm(a)]
    \item
    There exists an exact sequence $0\to L\to M\to N\to0$ of $R$-modules such that $L$ has finite length and $N$ is maximal Cohen--Macaulay.
    \item
    The module $M/{\Gamma_{\m}(M)}$ is maximal Cohen--Macaulay.
    \item
    The module $M$ is $d$-H-spherical.
    \item
    The module $M$ is locally maximal Cohen--Macaulay on the punctured spectrum of $R$.
\end{enumerate}
Then the following hold.
\begin{enumerate}[\rm(1)]
    \item
    The implications ${\rm(a)}\Leftrightarrow{\rm(b)}\Leftrightarrow{\rm(c)}\Rightarrow{\rm(d)}$ hold.
    \item
    If $d=1$, then the implication ${\rm(d)}\Rightarrow{\rm(c)}$ holds.
    \item
    If $d\ge2$, then the implication ${\rm(d)}\Rightarrow{\rm(c)}$ never holds.
\end{enumerate}
\end{prop}
\begin{proof}
(1) Suppose that ${\rm(a)}$ holds.
Since $\Gamma_{\m}(L)=L$ and $\Gamma_{\m}(N)=0$, the short exact sequence $0\to L\xrightarrow{f}M\to N\to 0$ yields an isomorphism $\Gamma_{\m}(f):L\xrightarrow{\sim}\Gamma_{\m}(M)$.
Hence an isomorphism $N\xrightarrow{\sim}M/\Gamma_{\m}(M)$ is induced, and ${\rm(b)}$ holds.
Conversely, since the module $\Gamma_{\m}(M)$ has finite length, if ${\rm(b)}$ holds, then the short exact sequence $0\to\Gamma_{\m}(M)\to M\to M/\Gamma_{\m}(M)\to0$ satisfies the required condition in ${\rm(a)}$.
We get the equivalence ${\rm(a)}\Leftrightarrow{\rm(b)}$.
Next, we note that $\Gamma_{\m}(M/\Gamma_{\m}(M))=0$ and $\H^i_{\m}(M)\cong \H^i_{\m}(M/\Gamma_{\m}(M))$ for all $i>0$; see \cite[Chapter 2]{BS} for instance.
The equivalence ${\rm(b)}\Leftrightarrow{\rm(c)}$ follows from these and Grothendieck's vanishing theorem.
As $M$ and $M/\Gamma_{\m}(M)$ are locally isomorphic on the punctured spectrum of $R$, the implication ${\rm(b)}\Rightarrow{\rm(d)}$ clearly holds.

(2) When $d=1$, the module $M/\Gamma_{\m}(M)$ is maximal Cohen--Macaulay for all finitely generated $R$-modules $M$ as $\Gamma_{\m}(M/\Gamma_{\m}(M))=0$, and therefore the assertion follows.

(3) The maximal ideal $\m$ of $R$ is locally isomorphic to $R$ on the punctured spectrum of $R$.
However, one has $\Gamma_{\m}(R)=\H^1_{\m}(R)=0$ by the assumption.
The long exact sequence
$$
0\to\Gamma_{\m}(\m)\to\Gamma_{\m}(R)\to\Gamma_{\m}(R/{\m})\to \H^1_{\m}(\m)\to \H^1_{\m}(R)\to\cdots
$$
indicates that $\H^1_{\m}(\m)\cong\Gamma_{\m}(R/{\m})=R/{\m}\ne0$.
Hence the implication ${\rm(d)}\Rightarrow{\rm(c)}$ never holds.
\end{proof}

Applying Theorem \ref{sphthm} to the stable category $\underline{\tf_n(R)}$ gives rise to the following theorem.
Here, $\fl(R)$ and $\sph_n^{\sf H}(R)$ respectively stand for the subcategory of $\mod R$ consisting of $R$-modules of finite length, and the subcategory of $\mod R$ consisting of $n$-H-spherical $R$-modules.

\begin{thm}\label{d-1}
Suppose that $R$ is Cohen--Macaulay and with dimension $d>0$.
\begin{enumerate}[\rm(1)]
    \item
    If $R$ is regular, then one has the equivalence
    $$
    \xymatrix@R-1pc@C-1pc{
    \underline{\tf_{d-1}(R)}\ar@<.5mm>[rrrr]^-{\Ext^d(\tr(-),R)}&&&&
    \fl(R).\ar@<.5mm>[llll]^-{\syz^{d-1}}
    }
    $$
    \item
    If $R$ is Gorenstein, then the following equalities hold.
    $$
    \sph_d^{\sf H}(R)=\fl(R)\ast\cm(R)=\{M\in\mod R \mid M/\Gamma_{\m}(M) \text{ is maximal Cohen--Macaulay }\}.
    $$
    Moreover, one then has the equivalence
    $$
    \xymatrix@R-1pc@C-1pc{
    \underline{\tf_{d-1}(R)}\ar@<.5mm>[rrrr]^-{\tr\syz^{d-1}\tr}&&&&
    \underline{\sph_d^{\sf H}(R)}.\ar@<.5mm>[llll]^-{\syz^{d-1}}
    }
    $$
\end{enumerate}
\end{thm}
\begin{proof}
Note that the equality $\grd_d(R)=\fl(R)$ holds; see \cite[Proposition 1.2.10]{BH}.
By \cite[Theorem 2.2.7]{BH} and \cite[Theorem 4.20]{AB}, if $R$ is regular (resp. Gorenstein), then any finitely generated $R$-module has projective (resp. Gorenstein) dimension at most $d$.
When this is the case, $\sph_{d}(R)$ (resp. $\sph_d^{\sf G}(R)$) is equal to $\G_{d-1,0}$.
Moreover, if $R$ is Gorenstein, then maximal Cohen--Macaulay modules are totally reflexive, and the converse is also true; see \cite[Theorem 3.3.10]{BH} for instance. 
Since one has the duality $\xymatrix@R-1pc@C-1pc{\underline{\G_{d-1,0}}\ar@<.5mm>[rr]^-{\tr}&&\underline{\tf_{d-1}(R)}\ar@<.5mm>[ll]^-{\tr}}$ by \cite[Proposition 1.1.1]{I}, the assertion follows from Theorem \ref{sphthm} and Proposition \ref{cm}.
\end{proof}

\begin{proof}[Proof of Corollary \ref{main2}]
The assertion of Corollary \ref{main2} is included in Theorem \ref{d-1}.
\end{proof}

We denote by $\Ref(R)$ the subcategory of $\mod R$ consisting of reflexive $R$-modules.
The exact sequence $0\to\Ext_{R^{\rm op}}^1(\tr M,R)\to M\xrightarrow{\sigma_M}M^{\ast\ast}\to\Ext_{R^{\rm op}}^2(\tr M,R)\to0$ for each finitely generated $R$-module $M$ shows that $M$ is reflexive if and only if it is 2-torsionfree.
In other words, the equality $\Ref(R)=\tf_2(R)$ holds.
The following corollary is a special case of Theorem \ref{d-1}.

\begin{cor}\label{refl}
Let $R$ be a three dimensional regular local ring.
One then has the equivalence
    $$
    \xymatrix@R-1pc@C-1pc{
    \underline{\Ref(R)}\ar@<.5mm>[rrrr]^-{\Ext^1((-)^\ast,R)}&&&&
    \fl(R).\ar@<.5mm>[llll]^-{\syz^{2}}
    }
    $$
\end{cor}
\begin{proof}
Since $(-)^\ast\approx\syz^2\tr(-)$, the assertion is none other than the case $d=3$ of Theorem \ref{d-1}(1).
\end{proof}

%%%%%%%%%%%%%%%%%%%%%%%%%%%%%%%%%%%%%%%%%
\begin{ac}
The author would like to thank his supervisor Ryo Takahashi for a lot of valuable discussions and advice.
\end{ac}
%%%%%%%%%%%%%%%%%%%%%%%%%%%%%%%%%%%%%%%%%

\end{document}